\documentclass[11pt]{amsart}
\usepackage{amssymb,verbatim,graphicx}
\usepackage{graphics,epsfig,psfrag} 

\newtheorem{theorem}{Theorem}[section]

\newtheorem{corollary}[theorem]{Corollary}
\newtheorem{question}[theorem]{Question}

\newtheorem{proposition}[theorem]{Proposition}
\newtheorem{lemma}[theorem]{Lemma}
\newtheorem{lem}[theorem]{}
\theoremstyle{definition}
\newtheorem{definition}[theorem]{Definition}
\theoremstyle{remark}
\newtheorem{remark}[theorem]{Remark}
\newtheorem{example}[theorem]{Example}
\newcommand{\blem}{\begin{lem} \rm}
\newcommand{\elem}{\end{lem}}

\newcommand{\R}{\mathbb{R}}

\newcommand{\C}{\mathbb{C}}

\newcommand{\Z}{\mathbb{Z}}
\newcommand{\Q}{\mathbb{Q}}

\renewcommand{\P}{\mathbb{P}}

\newcommand\lie[1]{\mathfrak{#1}}

\newcommand{\g}{\lie{g}}

\renewcommand{\t}{\lie{t}}

\newcommand{\on}{\operatorname}

\newcommand{\orb}{\on{orb}}
\newcommand{\mfd}{\on{mfd}}

\newcommand{\dual}{\vee}

\newcommand{\Aut}{ \on{Aut} }

\newcommand{\ssm}{-}

\newcommand\dirac{/\kern-1.2ex\partial} % Dirac operator
\newcommand\qu{/\kern-.7ex/} % Categorical quotients
\newcommand\lqu{\backslash \kern-.7ex \backslash} % Categorical
\newcommand\dr{r_+ \kern-.7ex - \kern-.7ex r_-}
 
\newcommand{\labell}\label

\renewcommand{\d}{{\on{d}}}
\newcommand{\ol}{\overline}

\newcommand\Phinv{\Phi^{-1}}

\newcommand\eps{\epsilon}

\newcommand{\lan}{\langle}
\newcommand{\ran}{\rangle}

\newcommand{\ti}{\tilde}

\newcommand\bra[1]{ < \kern-.7ex {#1} \kern-.7ex >} % Categorical quotients
\newcommand\bdefn{\begin{definition}}
\newcommand\edefn{\end{definition}}
\newcommand\bea{\begin{eqnarray*}}
\newcommand\eea{\end{eqnarray*}}
\newcommand\bcv{\left[ \begin{array}{r} }
\newcommand\ecv{\end{array} \right] }

\newcommand\bma{\left[ \begin{array}{l} }
\newcommand\ema{\end{array} \right]}
\newcommand\ben{\begin{enumerate}}
\newcommand\een{\end{enumerate}}
\newcommand\beq{\begin{equation}}
\newcommand\eeq{\end{equation}}
\newcommand\bex{\begin{example}}
\newcommand\bsj{\left\{ \begin{array}{rrr} }
\newcommand\esj{\end{array} \right\}}

\newcommand\eex{\end{example}}

\newcommand\sx{*\kern-.5ex_X}

\def\mathunderaccent#1{\let\theaccent#1\mathpalette\putaccentunder}
\def\putaccentunder#1#2{\oalign{$#1#2$\crcr\hidewidth \vbox
to.2ex{\hbox{$#1\theaccent{}$}\vss}\hidewidth}}

\begin{document}

\title[Quasimap Floer cohomology for varying symplectic quotients]{Quasimap Floer
  cohomology for varying symplectic quotients}

\author{Glen Wilson} 

\address{Mathematics-Hill Center,
Rutgers University, 110 Frelinghuysen Road, Piscataway, NJ 08854-8019,
U.S.A.}  \email{glenmatthewwilson@gmail.com}

\author{Christopher T. Woodward}

\thanks{Partially supported by NSF
 grant DMS0904358 and the Simons Center for Geometry and Physics}

%\date{2/3/12} 

\address{Mathematics-Hill Center,
Rutgers University, 110 Frelinghuysen Road, Piscataway, NJ 08854-8019,
U.S.A.}  \email{ctw@math.rutgers.edu}

%C: added abstract
\begin{abstract}   
We show that quasimap Floer cohomology for varying symplectic
quotients resolves several puzzles regarding displaceability of toric
moment fibers.  For example, we (i) present a compact Hamiltonian
torus action containing an {\em open} subset of non-displaceable
orbits and a codimension four singular set, partly answering a
question of McDuff, and (ii) determine displaceability for most of the
moment fibers of a symplectic ellipsoid.  
\end{abstract}

\maketitle

\tableofcontents

\section{Introduction} 

{\em Quasimap Floer cohomology}, constructed in \cite{wo:gdisk}, is an
obstruction to Hamiltonian displaceability of an invariant Lagrangian
submanifold in the zero level set of a moment map for the action of a
Lie group by an invariant time-dependent Hamiltonian.  The
differential for quasimap Floer cohomology counts orbits of the group
on the space of holomorphic disks with boundary in the Lagrangian.
Since holomorphic disks ``upstairs'' often have better properties than
holomorphic disks in the symplectic quotient, quasimap Floer
cohomology is a sometimes-better-behaved substitute for Floer
cohomology in the quotient.

Here we restrict to the case of toric varieties.  That is, the group
$G \subset U(1)^N$ is a torus and the symplectic manifold $Y \cong
\C^N$ is a Hermitian vector space.  The group $G$ acts in Hamiltonian
fashion on $Y$ with quadratic moment map $\Psi: Y \to \g^\dual$.  The
symplectic quotient $X = Y \qu G := \Psi^{-1}(0)/G$ is a possibly
singular toric manifold with action of the torus $T = U(1)^N/G$ and a
moment map $\Phi: X \to \t^\dual$ induced from that of $U(1)^N$ on
$Y$.  By definition, a smooth function on $X$ is an equivalence class
of smooth $G$-invariant functions on $Y$, so that displaceability in
$X$ is equivalent to displaceability by a $G$-invariant Hamiltonian on
$Y$.  Non-displaceability results in the quotient $X$ are provided by
Floer-theoretic methods in Fukaya-Oh-Ohta-Ono \cite{fooo:toric1},
\cite{fooo:toric2}.  Quasimap Floer cohomology gives the following
result, which at first seems only slightly stronger.  We denote by
$v_1,\ldots,v_N \in \t$ the images of minus the standard basis vectors
$e_1,\ldots,e_N \in \R^N$.  The moment polytope $\Phi(X)$ is the set
of points satisfying linear inequalities
\begin{equation} \label{lineq} 
\Phi(X) = \{ \lambda \in \t^\dual \ | \ l_i(\lambda) \ge 0 \} , \quad
l_i(\lambda)/2\pi := \langle \lambda, v_i \rangle -\eps_i, \quad i =
1,\ldots, N \end{equation}
where $\lan \cdot, \cdot \ran: \t^\dual \times \t \to \R$ is the
canonical pairing and $\eps_1,\ldots,\eps_N$ are constants given by
the choice of moment map.  This list of inequalities may not be
minimal, that is, any particular inequality may or may not define a
facet of $\Phi(X)$.  Let $\Lambda$ be the {\em universal Novikov
  field} consisting of possibly infinite sums of real powers of a
formal variable $q$,
$$ \Lambda = \left\{ \sum_{n =0}^\infty c_n q^{d_n} , \quad c_n \in
\C, d_n \in \R, \quad \lim_{n \to \infty} d_n = \infty \right\} .$$
Let $\Lambda_0$ denote the subring consisting of sums with only
non-negative powers.  Any fiber $L_\lambda = \Phi^{-1}(\lambda)$ over
an interior point $\lambda \in \on{int}(\Phi(X))$ is a Lagrangian
torus in $X$, namely a single free $T$-orbit, and has inverse image
$\ti{L}_\lambda$ in $\Psi^{-1}(0)$ a $U(1)^N$-orbit in $Y$.  We
identify
$H^1(L_\lambda,\Lambda_0) \cong H^1(T,\Lambda_0)^T \cong \t^\dual \otimes \Lambda_0.$
In particular for any $v \in \t$ and $b \in H^1(L_\lambda,\Lambda_0)$
we have a pairing $\lan v, b \ran \in \Lambda_0$ and an exponential
$e^{\lan v, b \ran} \in \Lambda_0$.  Choose 
%G \Lambda_0 -> \R  C: oops ..... it should be added to the exponential
$\delta = (\delta_1,\ldots,\delta_N) \in \Lambda_0^N$.  The {\em
  bulk-deformed potential} (introduced in \cite[Theorem 3]{giv:hom})
is
\begin{equation} \label{hv}
 W_{\lambda,\delta}: H^1(L_\lambda,\Lambda_0) \to \Lambda_0, \quad b
 \mapsto \sum_{i=1}^N e^{\lan v_i, b \ran - \delta_i} q^{l_i(\lambda)} .\end{equation}

\begin{theorem} \label{main}   For any $\lambda \in \on{int}(\Phi(X)))$, 
if there exists $\delta \in \Lambda_0^N$ such that
$W_{\lambda,\delta}$ has a critical point, then $L_\lambda$ is
non-displaceable in $X$, or equivalently, $\ti{L}_\lambda$ is not
displaceable by any $G$-invariant time-dependent Hamiltonian $H \in
C^\infty([0,1] \times Y)^G$.
\end{theorem} 

This was proved in \cite{wo:gdisk} but the possibility that the
quotient is singular or that some of the inequalities do not define
facets of the polytope was not included in the main result.  Later we
realized the importance of the more general result: even for
understanding displaceability in open subsets of $\C^N$, the case of
singular or ``spurious'' inequalities is highly relevant.  The following
example shows the importance of singular quotients.

\begin{example}  (Non-displaceability in $\C^2$ by $\Z_2$-invariant
Hamiltonians) Let $\mu = (\mu_1,\mu_2) \in \R^2_{> 0}$, and $L = \{
(z_1,z_2) \in \C^2 \ | \ |z_1|^2 = \mu_1, |z_2|^2 = \mu_2 \}$.  The
group $\Z_2 = \{ \pm 1 \}$ acts diagonally on $\C^2$.  We claim that
$L \subset \C^2$ is displaceable by a $\Z_2$-invariant Hamiltonian iff
$\mu_1 \neq \mu_2$.  Indeed, $L$ is displaceable by a $\Z_2$-invariant
time-dependent Hamiltonian iff $L /\Z_2$ is displaceable in the
orbifold quotient $X = \C^2/\Z_2$.  The latter admits the structure of
a toric orbifold with moment polytope given by the span of the vectors
$(1,1),(-1,1)$, see Example \ref{c2z21} below and Figure
\ref{c2z2fig}.  The quotient $L/\Z_2$ is the moment fiber over
$(\lambda_1,\lambda_2) = (\mu_1 - \mu_2, \mu_1 + \mu_2)$.  The probes
of McDuff \cite{mc:di} (or the Hamiltonian action of $SU(2)$ on
$\C^2$) show that if $\lambda_1 \neq 0$, then $L$ is displaceable,
since any such $(\lambda_1,\lambda_2)$ is contained in a probe with
direction $(0,1)$.  Unfortunately quasimap Floer cohomology for
$\C^2/\Z_2$ does not give any non-displaceable fibers, since the
potential has no critical points.  We apply the following trick:
observe that $0 \in \C^2$ is fixed by the flow of any $\Z_2$-invariant
Hamiltonian $H$, since $dH(t,0) = 0 $ for all $t \in [0,1]$.  Consider
the singular symplectic quotient $\hat{X}$ obtained from $X \times \C$
by symplectic quotient by $S^1$, acting on $X$ with moment map
$|z_1|^2 + |z_2|^2$, that is, by symplectic cut of $X$ with respect to
the diagonal circle.  The space $\hat{X}$ has the same moment polytope
as $X$, but its realization as a symplectic quotient $Y \qu G$
involves a ``spurious inequality'' $\lambda_2 \ge 0$ as well as the
inequalities for $X$ given by $\lambda_1 + \lambda_2 \ge 0 , -
\lambda_1 + \lambda_2 \ge 0$.  A toric moment fiber for $X$ is
displaceable iff the corresponding toric moment fiber for $\hat{X}$ is
displaceable, since after applying a cutoff function we may assume
that $H$ vanishes near the singular locus and $X$ and $\hat{X}$ are
isomorphic away from the singular loci, see Proposition \ref{disp} below.
The bulk-deformed potential
$ q^{\lambda_1 + \lambda_2} e^{b_1 + b_2} + q^{- \lambda_1 +
  \lambda_2} e^{-b_1 + b_2} + q^{\lambda_2} e^{b_2 - \delta} $
has a critical point iff $\lambda_1 = b_1 = 0$ and $2e^{b_2} + e^{b_2
  - \delta} =0 $ or $e^{- \delta} = -2 $.  See Figure \ref{c2z2fig}.
A similar result for the deformation of $\C^2/\Z_2$ was studied in
Fukaya et al \cite{fukaya:toricdegen}.
\end{example} 

\begin{figure}[ht]
\includegraphics[height=.7in]{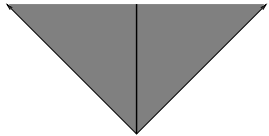}
\caption{Displaceable and non-displaceable fibers for $\C^2/\Z_2$.} 
\label{c2z2fig}
\end{figure} 

Below we give further examples of displaceability of torus orbits in
open subsets of $\C^N$.  Embedding such open subsets in singular
symplectic quotients turns out to by quite useful for resolving
displaceability.  We use the same technique to partially answer a
question of McDuff, by giving an example of a compact toric orbifold
with an open subset of non-displaceable fibers.

We thank M. S. Borman, D. McDuff, and K. Fukaya for helpful comments.

\section{Displaceability in orbifolds}

In this section we review some basic facts about Hamiltonian
displaceability in orbifolds.  Recall that an {\em orbifold} is a
Hausdorff second-countable topological space $X$ equipped with an
equivalence class of {\em orbifold structures}: a smooth proper
\'etale groupoid $\ti{X}$ together with a homeomorphism from the space
of isomorphism classes of objects in $\ti{X}$ to $X$, see
e.g. Adem-Klaus \cite{ad:lec}.  For any orbifold $X$ and element $x
\in X$, we denote by $\Aut(x)$ the group of automorphisms of any
object $\ti{x}$ in $\ti{X}$ mapping to $x$, independent up to
isomorphism of the choice of orbifold structure and choice of
$\ti{x}$.  Denote by $X^{\orb} = \{ x \in X | \# \Aut(x) > 1 \} $ the
subset of $X$ consisting of points with more than one automorphism and
by $X^{\mfd} = X \ssm X^{\orb} $ the locus of points with only the
identity morphism.  Thus $X^{\mfd}$ is a smooth manifold and admits
an open embedding into $X$.

Orbifolds typically arise as quotients of smooth manifolds by locally
free actions of compact groups.  The quotient $Y/G$ of a smooth
manifold $Y$ by a compact group $G$ has a canonical orbifold
structure, given by taking local slices for the action.  A $G$-space
$Y$ together with an orbifold equivalence $Y/G \to X$ is called a {\em
  global quotient presentation} of $X$.  If the generic automorphism
group of an orbifold $X$ is trivial, so that $X^{\mfd}$ is non-empty,
then $X$ admits a global quotient presentation, namely the orthogonal
frame bundle of $X$ by the action of the orthogonal group.  The notion
of an action of a Lie group on an orbifold $X$ is somewhat complicated
in general because of the various notions of an action of a Lie group
on a category, see e.g. \cite{lm:sdm}.  In this paper all group
actions will arise from global presentations, that is, from a
$G$-equivariant action on $Y$ where $X = Y/G$.  The notion of
symplectic form and Hamiltonian action have natural extensions to the
orbifold case, which are somewhat simpler in the globally presented
case: a symplectic form on a globally presented orbifold $X = Y/G$ is
a closed $G$-basic form on $Y$ that is non-degenerate on the normal
bundles to the $G$-orbits.

\begin{definition}  Let $X$ be a symplectic orbifold.  
A subset $L \subset X$ is {\em Hamiltonian displaceable} iff there
exists a function $H \in C^\infty_c([0,1] \times X)$ with time $t$
Hamiltonian flow $\phi_{H,t}: X \to X$ such that $\phi_{H,t}(L) \cap L
= \emptyset$ for some $t$.
\end{definition} 

We collect a few elementary properties of displaceability in the following. 

\begin{proposition} \label{disp}
\begin{enumerate} 
\item Suppose that $X_1 \subset X_2$ is an open set and $L \subset
  X_1$.  If $L$ is displaceable in $X_1$, then $L$ is displaceable in
  $X_2$.
\item Suppose that either $L_1 \subset X_1$ or $L_2 \subset X_2$ are
  displaceable.  Then $L_1 \times L_2$ is displaceable in $X_1 \times
  X_2$.
\item Suppose that $X$ is a Hamiltonian $G$-orbifold and $X \qu G$ its
  symplectic quotient.  Then $L \subset X \qu G$ is displaceable iff
  the inverse image of $L$ in $X$ is displaceable by the flow of a
  $G$-invariant time-dependent Hamiltonian. \label{dispquot}
\item Suppose that $L_1,L_2 \subset X$ are disjoint subsets such that
  $L_1$ is displaceable by a flow $\phi_{t,H}$ with $\phi_{t,H}(L_2) =
  L_2$ for all $t \in [0,1]$.  Then $L_1$ is displaceable by a flow
  $\phi_{t,H_2}$ equal to the identity on an open neighborhood of $L_2$ for
  all $t \in [0,1]$.
\end{enumerate} 
\end{proposition} 

\begin{proof}   (a) If $H_1 \in C^\infty_c(X_1)$ displaces $L_1$ in $X_1$, 
then the extension of $H_1$ by zero to $C^\infty_c(X_2)$ displaces
$L_1$ in $X_2$.  (b) Suppose without loss of generality that $H_1$
displaces $L_1$.  Then $\pi_1^* H_1$ displaces $L_1 \times L_2$ where
$\pi_1: X_1 \times X_2 \to X_1$ is the projection.  (c) If $H$
displaces $L$ then any invariant extension of $\pi^* H$, where $\pi:
\Phinv(0) \to X \qu G$, to $X$, displaces the inverse image of $L$.
The converse is similar.  (d) If $\phi_{t,H}$ displaces $L_1$ and maps
$L_2$ to itself, then let $\rho \in C^\infty(X)$ be a function equal
to $1$ on an open neighborhood of the image of $L_1$ under the flow
$\phi_{H,t}$, and zero on an open neighborhood of $L_2$.  Then the
flow of $\rho H$ displaces $L_1$ and is equal to the identity on a
neighborhood of $L_2$.
\end{proof} 

\section{Displaceability of toric moment fibers} 

We consider the following
class of possibly non-compact Hamiltonian torus actions on orbifolds.
Let $T$ be a torus and $\t_\Z = \exp^{-1}(1)$ the integral lattice.

\begin{definition} \label{opendef}   $X$ is an {\em open symplectic toric orbifold} 
for $T$ if $X$ is a connected Hamiltonian $T$-orbifold with moment map
$\Phi: X \to \t^\dual$ satisfying the following conditions:
\begin{enumerate}
\item $\Phi(X)$ is a defined by a finite set of affine linear
  inequalities defined by vectors in $\t_\Z$ and strict affine linear
  inequalities defined by vectors in $\t$;
\item $\Phi: X \to \Phi(X)$ is proper;
\item the $T$-action is generically free;
\item $\dim(T) = \dim(X)/2$.
\end{enumerate}
 \end{definition}

\begin{remark}  By item (a) the image $\Phi(X)$ is given by 
\begin{equation} \label{defined} 
 \Phi(X) = \left\{ \lambda \in \t^\dual \left| \begin{array}{ll} \lan
   \lambda,v_i \ran \ge \eps_i & i = 1,\ldots, k, \\ \lan \lambda,v_i
   \ran > \eps_i & i = k+1,\ldots, N \end{array} \right. \right\}
 \end{equation}
for some vectors $v_i \in \t_\Z, i = 1,\ldots k$ and $v_i \in \t, i =
k+1,\ldots, N$.  By items (b),(c) and the results of \cite{le:ha}, the
stabilizer of any point in the inverse image $\Phinv(F)$ of an open
facet $F$ is isomorphic to the cyclic group $\Z_{n(F)}$ for some
integer $n(F) \ge 1$.  We assume that $v_i$ is normalized to be the
$n(F_i)$-th multiple of the primitive lattice vector pointing inward
from the facet $F_i$ corresponding to $v_i$.  In this way, in the
compact case the vectors $v_i$ are the data used in the {\em weighted
  fan} classification of toric orbifolds in \cite{le:ha},
\cite{bcs:tdms}.
\end{remark} 

\begin{example} $X = \C^n$ itself is an open symplectic toric manifold
with symplectic form $-2 \sum_{i=1}^n \d q_j \wedge d p_j $ where $z_j
= q_j + i p_j$ and moment map $\Phi: X \to \t^\dual \cong \R^n,
(z_1,\ldots,z_n) \mapsto (|z_1|^2,\ldots, |z_n|^2)$.  The vectors
$v_1,\ldots, v_n$ are the standard basis vectors.
\end{example} 

\begin{example}  Any open subset of $\C^n$ defined by 
$ a_1 |z_1|^2 + \ldots + a_n |z_n|^2 < 1$ for some $ a_1,\ldots,a_n \in
  \Q $
is an open symplectic toric manifold obtained from a weighted
projective space by removing a divisor at infinity. 
\end{example} 

The following well-known lemma indicates how to read off the
automorphism group $\Aut(x)$ of a point $x \in X$ from the facets of
$\Phi(X)$ containing $\Phi(x)$.

\begin{lemma}  (see e.g. 
\cite[Lemma 6.2]{le:ha}) Let $X$ be an open symplectic toric orbifold,
$x \in X$, and $I(x) = \{ i | \lan \Phi(x), v_i \ran = \eps_i \}$ the
indices of normal vectors of facets containing $\Phi(x)$.  Then
$$ \Aut(x) \cong \on{ker} \left( U(1)^{\#I(x)} \to T, \exp\left( \sum_{i
  \in I(x)} c_i e_i\right) \mapsto \exp\left(\sum_{i \in I(x)} c_i v_i\right)
\right) .$$
\end{lemma} 

In particular since $\Aut(x)$ is finite this lemma implies that
$\Phi(X)$ is a {\em simple} polytope, that is, the normal vectors at
any point are linearly independent.

\begin{example} \label{c2z21}
 Let $X$ be the quotient of $Y = \C^2$ by the diagonal action of $\Z_2
 = \{ \pm 1\}$ given by scalar multiplication. The action of $T' =
 U(1)^2$ on $\C^2$ descends to a generically free action of $T =
 T'/\Z_2$ on $X$.  The integral lattice $\t_\Z $ is the inverse image
 of $\Z_2$ under the exponential map for $T'$, hence $\t_\Z$ is
 generated by $(1/2,1/2), (1/2,-1/2)$.  We identify $\t \to \R^2$ by
 $(\xi_1,\xi_2) \mapsto (\xi_1 + \xi_2, \xi_1 - \xi_2)$ so the
 integral lattice becomes the standard one.  The moment map for the
 $T$ action on $X$ is then
$ \Phi: X \to \R^2, \quad (z_1,z_2) \to (|z_1|^2 + |z_2|^2, |z_1|^2 -
|z_2|^2) .$
The integral vectors are $(-1,1),(1,1)$.  The automorphism group
$$\Aut(0) = \on{ker}(U(1)^2 \to U(1)^2, (z_1,z_2) \mapsto
(z_1z_2,z_1z_2^{-1})) = \Z_2 .$$
\end{example} 

\begin{theorem}   Any open symplectic toric orbifold can be obtained
by symplectic reduction by an open subset of $T^* U(1)^k \times \C^l$
for some $k,l$ by a subtorus of $ U(1)^{k+l}$.
\end{theorem} 

\begin{proof} The compact case is consequence of the results of 
\cite{le:nc} as discussed in \cite{wo:gdisk}: Given an open symplectic
toric orbifold $X$ the symplectic cutting construction constructs a
symplectic toric orbifold $X'$ with the same moment polytope as $X$.
The uniqueness result of \cite{le:nc} implies that $X$ is isomorphic
to $X'$ as a Hamiltonian $T$-orbifold.
\end{proof} 

We wish to understand the Hamiltonian displaceability of {\em toric
  moment fibers}:

\begin{definition}  Let $X$ be an open symplectic toric orbifold.  
A {\em toric moment fiber} is a Lagrangian torus given as a fiber
$L_\lambda = \Phinv(\lambda)$ for some $\lambda \in \on{int}(\Phi(X))$.
\end{definition}

Denote by $ND(X) \subset \on{int}(\Phi(X))$ resp. $D(X)$ the set of
points corresponding to non-displaceable resp. displaceable toric
moment fibers.

\begin{example} (Moser)  Let $X$ be the unit disk 
with moment polytope $[0,1)$.  Then $D(X) = [0,1/2)$ and $ND(X) =
    [1/2,1)$.  For Moser \cite{mo:vol} shows that the only invariant
      of a symplectic surface is its area.  Hence a circle $L$ in the
      disk $X$ encloses less than half the area iff $L$ is
      displaceable in $X$.  Similarly, if $X = \P^1$ with moment
      polytope $\Phi(X) = [-1,1]$ then $ND(X) = \{ 0 \}$.
\end{example} 

\begin{example}\label{probes} (McDuff \cite{mc:di})  Let $X$ be a compact symplectic toric 
orbifold with moment map $\Phi$, and let $F$ be an open facet of
$\Phi(X)$ such that $\Phinv(F) \subset X^{\mfd}$.  Let $v \in
\t^\dual_\Z$ be a vector such that $v$ can be completed to a lattice
basis by vectors parallel to $F$.  If $\lambda_0 \in F$ and $\lambda$
lies less than half-way along $(\lambda_0 + \R_{\ge 0 }v ) \cap
\Phi(X)$, then $\Phinv(\lambda)$ is displaceable.  For let $T_0
\subset T$ be the torus whose Lie algebra is the annihilator of $v$.
Moser's argument shows that $\Phinv(\lambda)/T_0$ is displaceable in
$X \qu T_0$, and then \eqref{dispquot} of Proposition \ref{disp}
implies that $\Phinv(\lambda)$ is displaceable in $X$.  See
Abreu-Borman-McDuff \cite{abreu:ext} for improvements on this method.
\end{example} 

Naive application of Theorem \ref{main} (that is, without spurious
inequalities) does not come close to resolving the questions of
non-displaceability of toric fibers even for simple examples and after
including bulk deformations in \cite{fooo:toric2}.  For example, for a
weighted projective space the naive method gives a single
non-displaceable fiber over $\lambda = (5/3,5/3)$, while McDuff's
method shows displaceability for only some of the other fibers. See
Example \ref{p135more} below.

\section{Potentials for varying quotients} 

As explained in the introduction, open symplectic toric manifolds have
various realizations as symplectic quotients, some singular, and the
quasimap Floer cohomology for each realization can give additional
information about displaceability.  We combine the potentials for the
different compactifications into a potential involving infinitely many
variables as follows.

\begin{definition}  An affine linear function
$\ell: \t^\dual \to \R$ is {\em semipositive} on $\Phi(X)$ iff $\ell$
  is positive on $\Phi(X^{\mfd})$ and non-negative on
  $\Phi(X^{\orb})$.
\end{definition} 

\begin{remark}  Any affine linear function $\ell$ on $\t^\dual$ is given by
$ \ell(\lambda) = \lan v, \lambda \ran - \eps $ for some $v \in \t, \eps \in \R$.  If the function corresponding to
$\lambda,\eps$ is semipositive then so is the function corresponding
to $\lambda, \eps'$ for any $\eps' \leq \eps$.
\end{remark} 

\begin{example}  Let $X = \P(1,1,2)$ denote the weighted
projective plane with moment map the convex hull of $(0,0), (1,0)$ and
$(0,2)$, with the orbifold singularity with automorphism group $\Z_2$
mapping to $(1,0)$.  Then the linear function $ \lan ( -1,0), \cdot
\ran - \eps$ is semipositive for $\eps \leq -1$, while the linear
function $ \lan (0,-1), \cdot \ran - \eps$ is semipositive for $\eps <
-2$.

\end{example} 

\begin{definition} 
Denote by $C(\t_\Z,\ol{\R})_+$ the set of maps $\eps: \t_\Z \to \R
\cup \{ - \infty \}$ such that 
\begin{enumerate}
\item only finitely many values of $\eps$ are
finite;
\item if $v \in \t_\Z$ defines a facet of $\Phi(X)$ in the sense of
  \eqref{defined} then 
$\eps(v) = \min_{\lambda \in \Phi(X)} \lan v, \lambda \ran ;$
\item if $v \in \t_\Z$ does not define a facet then $\lan v, \cdot
  \ran - \eps(v)$ is semipositive on $\Phi(X)$.
\end{enumerate} 
\end{definition} 

\begin{definition} 
The {\em potential} for
$ \lambda \in \on{int}(\Phi(X)), \ \eps \in C(\t_\Z,\ol{\R})_+,
\ \delta \in C(\t_\Z,\Lambda_0)$
is the function
$$ W_{\lambda,\eps,\delta}: H^1(T,\Lambda_0) \to \Lambda_0, \quad b
\mapsto 
\sum_{v
  \in \t_\Z} q^{ \lan v,\lambda \ran - \eps(v)}e^{ \lan v,b \ran - \delta(v)} $$
where by convention $q^\infty = 0 $.  
\end{definition}  

\begin{example}[Symplectic balls] \label{balls}   
Let $X = \{z \in \C^n| \sum_{i=1}^n |z_i|^2 < 1 \}$ be the unit ball in
$\C^n$.  Consider the coweight $v = (-1,\ldots, -1)$ and let $\eps(v')
= - c$ if $v = v'$, $\eps(e_i)=0$ for all $1\leq i \leq n$ (where $e_i$
denotes the standard basis vector) and $\eps(v') = -\infty$ otherwise,
and $\delta(e_i) = 0 $.  Then
$$ W_{\lambda,\eps,\delta}(b) = \sum_{i=1}^n q^{\lambda_i} e^{b_i} +
q^{-\lambda_1 - \ldots - \lambda_n + c} e^{-b_1 - \ldots - b_n} $$
for $c \ge 1$. 
\end{example} 

\begin{theorem} \label{nondisp}  Suppose that $\lambda \in \on{int}(\Phi(X))$
is such that for some $\eps,\delta$, $W_{\lambda,\eps,\delta}$ has a
critical point.  Then $\Phinv(\lambda) \subset X$ is non-displaceable.
 \end{theorem}

Before we give the proof, we present several examples showing
how this Theorem improves on that of \cite{wo:gdisk}. 

\begin{example}[Symplectic balls continued]  Continuing
Example \ref{balls}, $W_{\lambda,\eps,\delta}$ has a critical point
iff $\lambda = (c,\ldots,c)/(n+1)$ for $c \ge 1$ iff $\Phinv(\lambda)$
is non-displaceable which is well-known from the works of Cho
\cite{cho:hol} and Entov-Polterovich \cite{entov:rigid}.  McDuff's
method implies that the remaining toric fibers are displaceable.  See
Figure \ref{ball}.
\end{example} 

\begin{figure}[ht]
\includegraphics[height=1in]{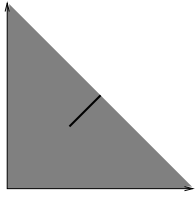}
\caption{Displaceable and non-displaceable fibers for the symplectic
  $4$-ball.}
\label{ball}
\end{figure} 

\begin{example} [A weighted projective plane with a measure zero
set of non-displaceable fibers] Suppose that $X = \P(1,1,2)$ is the
  weighted projective plane with moment polytope $(0,0),(1,0),(0,2)$.
  We write
$$ \Phi(X) = \{ (\lambda_1,\lambda_2) | \lambda_1 \ge 0, \lambda_2 \ge
0, 2\lambda_1 + \lambda_2 \leq 2, \lambda_1 \leq 1 \} .$$
The corresponding potential is 
$$ W_{\lambda,\eps,\delta}(b) = q^{\lambda_1} e^{b_1} +
q^{\lambda_2}e^{b_2} + q^{-2 \lambda_1 - \lambda_2 + 2} e^{-2b_1 -
  b_2} + q^{- \lambda_1 - \eps} e^{-b_1 - \delta_1} .$$
For $\eps = 1$ we obtain a critical point iff $\lambda = (1,0) +
\zeta(-1,1)$ where $\zeta \leq 1/2$.  See Figure \ref{p112}.
\end{example} 

\begin{figure}[ht]
\includegraphics[height=1.7in]{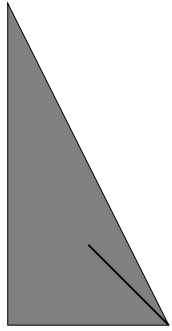}
\caption{Displaceable and non-displaceable fibers for $\P(1,1,2)$.}
\label{p112}
\end{figure} 

\begin{example} [A symplectic ellipsoid] Suppose that 
$$X =  \{ (z_1,z_2) \in \C^2 |  |z_1|^2 + |z_2|^2/2 < 1 \}.$$ 
We write
$$ \Phi(X) = \{ (\lambda_1,\lambda_2) | \lambda_1 \ge 0, \lambda_2 \ge
0, 2\lambda_1 + \lambda_2  <  2, \lambda_1 < 1 \} .$$
For $\eps_1, \eps_2 \geq 0$, the corresponding potential is
$$ W_{\lambda,\eps,\delta}(b) = q^{\lambda_1} e^{b_1} +
q^{\lambda_2}e^{b_2} + q^{-2 \lambda_1 - \lambda_2 - \eps_1} e^{-2b_1 -
  b_2 - \delta_1} + q^{- \lambda_1 - \eps_2} e^{-b_1 - \delta_2} .$$
Then $W_{\lambda,\eps,\delta}$ has a critical point for some
$\eps,\delta$ iff $\lambda_1 + \lambda_2 \ge 1, \lambda_1 \ge 1/2$.
Namely $\eps_1 = 2\eps_2$ and $\eps_2 = - \lambda_1 - \lambda_2 $.
Most of the remaining fibers are displaceable by probes, although we
did not manage to resolve the question completely, see Figure
\ref{open12}.
\end{example}

\begin{figure}[ht]
\includegraphics[height=2in]{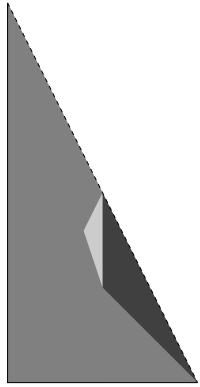}
\caption{Displaceable and non-displaceable fibers for the ellipsoid}
\label{open12}
\end{figure} 

\begin{example} [A weighted projective plane with a positive
measure subset of non-displaceable fibers] 
\label{p135more} 
We show that the toric
  orbifold $X = \P(1,3,5)$ contains an open subset of non-displaceable
  moment fibers.  This partly answers a question of McDuff who asked
  whether there is an example of such an action, presumably thinking
  of the smooth case.  The moment polytope is the convex hull $(0,0),
  (3,0), (0,5)$, and can be defined by the inequalities
$$ \Phi(X) = \{ \lambda \in \R^2, \lambda_1 \ge 0, \lambda_2 \ge 0,
5\lambda_1 + 3 \lambda_2 \leq 15, - \lambda_1 \geq - 3, -2 \lambda_1 -
\lambda_2 \geq -6 \} .$$
Consider the  potential 
\begin{multline}
 W_{\lambda,\eps_1,\eps_2,\delta_1,\delta_2}(b)= q^{\lambda_1} e^{b_1}
 + q^{\lambda_2} e^{b_2} + q^{-5 \lambda_1 - 3 \lambda_2 + 15} e^{-5
   b_1 - 3 b_2} \\ + q^{- \lambda_1 + 3 - \eps_1} e^{-b_1 - \delta_1} +
 q^{-2 \lambda_1 - \lambda_2 + 6 - \eps_2}e^{-2b_1 - b_2 - \delta_2}
 .\end{multline}
\begin{figure}[ht]
\includegraphics[height=3.5in]{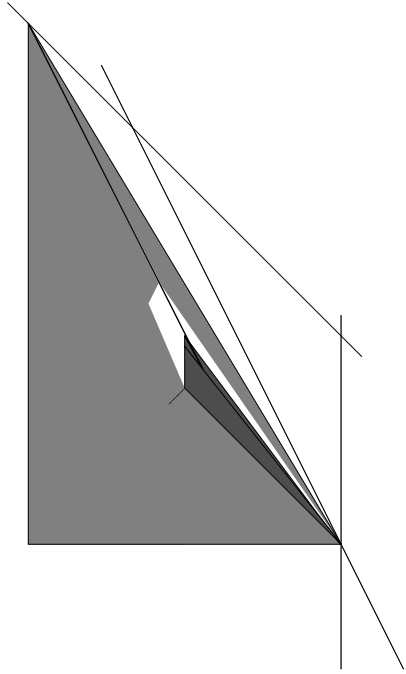}
\caption{More non-displaceable fibers for $\P(1,3,5)$ using spurious
inequalities}
\label{p35more}
\end{figure} 

The potential for $(\lambda_1,\lambda_2) = (3,0)+ c_1(-1,1) +
c_2(-3,4)$ has a critical point if $c_1,c_2 \ge 0$ but $ \langle
(\lambda_1,\lambda_2), (-1,0) \rangle + 3 \leq \langle
(\lambda_1,\lambda_2), (1,0) \rangle$; this is the condition that that
the terms defined by the facets with normal vectors
$(-1,0),(0,1),(-2,-1)$ have leading order terms that of equal order
and lower order than the terms arising from the remaining facets.  As
in the previous examples, this means that the potential arising from
these terms has a non-degenerate critical point.  The additional terms
do not affect the existence of a critical point, by \cite[Lemma
  10.16]{fooo:toric1} (which is a version of the implicit function
theorem for formal functions with values in the Novikov ring).  Note
that \cite[Lemma 10.16]{fooo:toric1} is written for integral polytopes
(polytopes corresponding to smooth toric varieties) but the technique
works equally well for arbitrary potentials, since integrality of the
basis given by the normal vectors at a vertex is never used in the
proof.
%
%C: 
%
%$ \langle (2,1), (-5,-3) \rangle + 15 = 2$ , $\langle (2,1),(-2,-1)
%\rangle + 6 = 1,$ $\langle (2,1),(-1,0) \rangle + 3 = 1$, and $\langle
%(2,1), (0,1) \rangle = 1 $
%%
%it follows that 
%
It follows that $\P(1,3,5)$ has an open subset of non-displaceable
fibers.  An additional line segment of non-displaceable fibers is
determined by the equality of powers in the leading order of terms
from the facets with normal vectors $(1,0), (-5,-3)$ and a
``spurious'' facet with normal vector $(-1,-1)$.  Additional open
region of non-displaceable torus fibers in $\P(1,3,5)$ are determined
by the leading order terms of (i) the facet with normal vector $(0,1)$
and the spurious facets with normal vectors $(-1,-1)$ and $(-1,0)$,
(ii) the facet with normal vector $(-5,-3)$ and the spurious facets
with normal vectors $(-1,0), (-2,-1)$, which were pointed out to us by
M. S. Borman, see \cite{abreu:ext}.  See Figure \ref{p35more}, where
the regions displaceable by McDuff's probes are shaded in lighter
grey.  Note that the projective line $\P(1,2)$ also has an open subset
of non-displaceable fibers as explained in \cite{wo:gdisk}, but this
is somewhat more expected since displaceability in $\P(1,2)$ is
equivalent to dispaceability in the disk and has singularities in
codimension $2$, not $4$.
 \end{example} 

\begin{proof}[Proof of Theorem \ref{nondisp}]    
First suppose that $X$ is a compact manifold, so that all of the
additional affine linear functions are strictly positive on $\Phi(X)$.
Suppose that $W_{\lambda,\eps,\delta}$ has a critical point for some
$\eps = (\eps(v))$.  Then $X$ is a symplectic quotient of the
representation $Y$ by a torus given as the kernel $G$ of the
homomorphism $U(1)^N \to T$ defined by the matrix formed by the
vectors $v \in \t_\Z$ where $\eps(v) \neq - \infty$.  The theorem then
follows from \cite[Theorem 7.1]{wo:gdisk}.

Next consider the case that $X$ is a compact orbifold.  Suppose that
$W_{\lambda,\eps,\delta}$ has a critical point for some $\eps =
(\eps(v))$.  Let $Y \qu G$ be the symplectic quotient of the
representation $Y$ as in the previous paragraph.  Then $Y \qu G$ is a
Hamiltonian $T$-orbifold on the locus where $G$ acts freely, and the
singular set of $Y \qu G$ (which can be worse than orbifold) is
contained in the singular set of $X$.  The proof of \cite[Theorem
  7.1]{wo:gdisk} shows that there is no $G$-invariant Hamiltonian on
$Y$ displacing the inverse image of $L_\lambda$ in $Y$.  On the other
hand, suppose that $L_\lambda$ is displaceable in $X$ by some
Hamiltonian $H$.  Necessarily, the flow of $H$ preserves $X^{\orb}$,
so if $\phi_H(L_\lambda) := \cup_{t \in [0,1]} \phi_{H,t}(L_\lambda)$
is the flow-out then $\phi_H(L_\lambda)$ is disjoint from $X^{\orb}$.
Choose a cutoff function $\rho \in C^\infty(X)$ such that $\rho$ is
equal to $1$ on an open neighborhood of $\phi_H(L_\lambda)$, and
has support contained in $X^{\mfd}$.  Then the flow of $\rho H$ also
displaces $L_\lambda$.  Since $\rho H$ vanishes in a neighborhood of
the singular set, $\rho H$ lifts to a smooth function on $Y$ which
displaces the inverse image of $L_\lambda$.

Finally consider the case that $X$ is non-compact.  Then $X^{\mfd}$ is
an open subset of the space $Y \qu G$ defined in the previous
paragraph.  Suppose that $L_\lambda$ is displaced by the flow of some
$H \in C^\infty_c([0,1] \times X)$, and $W_{\lambda,\eps,\delta}$ has
a critical point for some $\eps = (\eps(v))$.  Then after choosing a
cutoff function $\rho$ as in the previous paragraph, $\rho H$ lifts to
a smooth invariant function on $Y$ displacing the inverse image of
$L_\lambda$, which is a contradiction.
\end{proof}

\section{Functoriality of the quasimap mirror} 

According to the philosophy of mirror symmetry, the {\em mirror} of a
symplectic orbifold $X$ should be a complex space with potential
function $W: X^\dual \to \Lambda_0$, so that the Fukaya category of
$X$ is equivalent to the derived category of matrix factorizations.
As explained in Fukaya et al \cite{fooo:toric1}, the mirror of a toric
variety is a quantum correction of a potential obtained by Givental
\eqref{hv}.  In this section, we describe the quasimap mirror
construction (which is a somewhat naive version of the mirror but
perhaps more useful for determining displaceability) as a
contravariant functor which behaves well with respect to inclusions,
which clarifies various aspects of the displaceability problem.

\begin{definition}   Let $X$ be a (possibly) open toric orbifold
in the sense of Definition \ref{opendef}.   The {\em quasimap mirror} for $X$
is the space
$$ X^\dual := \t^\dual \times C(\t_\Z,\ol{\R})_+ \times
C(\t_\Z,\Lambda_0) \times H^1(T,\Lambda_0)
$$
equipped with the potential
$W: X^\dual \to \Lambda_0, \quad (\lambda,\eps,\delta,b) \mapsto
W_{\lambda,\eps,\delta}(b) .$
\end{definition} 

Note that the work of Fukaya et al \cite{fooo:toric1},
\cite{fooo:toric2} shows the existence of a particular deformation of
the naive potential which has the properties predicted by mirror
symmetry, such as the correct number of critical points which the
definition above lacks.  However, as we saw in the previous section,
the above formulation is more useful for detecting displaceability.
The quasimap mirror also has good functoriality properties, parallel
to the properties of displaceable fibers listed in Proposition
\ref{disp}.  The following definition will be used in the theorem to
relate the mirror of an action with the mirror for a quotient:

\begin{definition}  For any sub-torus $T_0 \subset T$ define $\pi:
\t \to \t/\t_0$ to be the projection and
\begin{equation} \label{pistar} (\pi_* \eps)(v_0) = \min_{ \pi(v)= v_0} \eps(v),
\quad (\pi_* \delta)(v_0) = \sum_{\pi(v) = v_0, \eps(v) = (\pi_*
  \eps)(v_0)} \delta(v) .\end{equation}
\end{definition}

\begin{theorem}[Functorial properties of quasimap mirrors] \label{mirrorfunc}
\begin{enumerate} 
\item 
  If $X_1 \to X_2$ is an open embedding of toric orbifolds then
  $X_2^\dual$ embeds canonically in $X_1^\dual$.  If $X_1 \to X_2 \to
  X_3$ are open embeddings then $X_3^\dual \to X_1^\dual$ is the
  composition of $X_3^\dual \to X_2^\dual$ and $X_2^\dual \to
  X_1^\dual$.
\item If $X_1,X_2$ are open toric sub-orbifolds of a toric orbifold
  $X$ then $(X_1 \cup X_2)^\dual = X_1^\dual \cap X_2^\dual$ and $(X_1
  \cap X_2)^\dual = X_1^\dual \cup X_2^\dual$.
\item $ (X_1 \times X_2)^\dual = X_1^\dual \times X_2^\dual$.
\item If $T_0 \subset T$ is a subtorus then considering $X \qu T_0$ as
  a toric $T/T_0$ orbifold then the space obtained from $X^\dual$ by
  composition with $H^1(T/T_0,\Lambda_0) \to H^1(T,\Lambda_0)$ and
  composition with $\pi_*$ from \eqref{pistar} embeds into $(X \qu
  T_0)^\dual$.
\end{enumerate} 
\end{theorem} 

The proof is immediate from the definition of semipositivity, in
particular in the setting of (a) if $\ell$ is semipositive on
$\Phi(X_2)$ then $\ell$ is automatically semipositive on $\Phi(X_1)$.
Part (a) says that the quasimap mirror construction gives a
contravariant functor.  Part (d) is only an embedding, because some of
the true facets of $\Phi(X)$ will not define facets of $\Phi(X \qu
T_0)$, so the mirror of $X \qu T_0$ is in general larger than that
obtained from $X$.

The functorial properties of the quasimap mirror construction
translates into the following functorial properties of the
corresponding non-displaceable moment fibers.  Say that $L_\lambda
\subset X$ is {\em (quasimap) Floer non-displaceable} if
$W_{\lambda,\eps,\delta}$ has a critical point for some $\eps,\delta$.
Let $FND(X)$ denote the set of Floer non-displaceable Lagrangians.
The following is a consequence of Theorem \ref{mirrorfunc}:

\begin{corollary}[Functorial properties of Floer non-displaceable sets]
\begin{enumerate} 
\item 
For any open embedding $X_1 \to X_2$, $FND(X_1) \supset FND(X_2)$.
\item If $X_1,X_2$ are open subsets of a toric orbifold $X$ then $
  FND(X_1 \cup X_2) \subset FND(X_1) \cap FND(X_2)$ and 
$FND(X_1 \cap X_2) \supset FND(X_1) \cup FND(X_2)$. 
\item $ FND(X_1 \times X_2) = FND(X_1) \times FND(X_2)$.
\item If $T_0 \subset T$ is a subtorus then $FND(X \qu T_0)$ contains
  the intersection of $FND(X)$ with the fiber over $0$ under the map
  $\t^\dual \to (\t/\t_0)^\dual$.
\end{enumerate} 
\end{corollary} 

The importance of the last item was emphasized out to us by
Abreu-Macarini \cite{abreu:rem}.

Obviously one would like to know whether one can obtain the
non-displaceable set from a cover.  For example:

\begin{proposition}   Any compact symplectic toric orbifold has a canonical 
open cover indexed by the fixed point set $X^T$, given as follows: for
each $x \in X^T$, let $X(x)$ be the open symplectic toric orbifold
obtained from $X$ by removing all divisors not containing $x$.  Then
$X = \cup_{x \in X^T} X(x)$.
\end{proposition} 

Because the quasimap mirror construction is contravariant, one cannot
expect to ``recover'' the symplectic topology of a toric orbifold from
the symplectic topology of its canonical cover. Rather, the symplectic
topology of each open subset already ``knows'' about the symplectic
topology of the compactification.  Still one would like to know the
relationship between displaceability in $X$ and displaceability in the
open cover.  The following question is, as far as we know, open:

\begin{question} Is $ND(X) = \cap_{x \in X^T} ND(X(x)), \quad D(X) = \cup_{x \in X^T}
  D(X(x)) ?$
\end{question}

\def\cprime{$'$} \def\cprime{$'$} \def\cprime{$'$} \def\cprime{$'$}
  \def\cprime{$'$} \def\cprime{$'$}
  \def\polhk#1{\setbox0=\hbox{#1}{\ooalign{\hidewidth
  \lower1.5ex\hbox{`}\hidewidth\crcr\unhbox0}}} \def\cprime{$'$}
  \def\cprime{$'$}

\end{document}